\documentclass[a4paper]{article}
\usepackage{amsmath,amsthm,amssymb,enumerate,hyperref,bbm,color}
\usepackage[legacycolonsymbols]{mathtools}
\providecommand\given{}
\newcommand\SetSymbol[1][]{%
  \nonscript\:#1\vert
  \allowbreak
  \nonscript\:
  \mathopen{}}
\DeclarePairedDelimiterX\Set[1]{\{}{\}}{%
  \renewcommand\given{\SetSymbol[\delimsize]}
  #1}

\DeclarePairedDelimiterXPP\pospart[1]{}{(}{)}{^+}{#1}
\DeclarePairedDelimiterXPP\negpart[1]{}{(}{)}{^-}{#1}

\newcommand\m{\phantom{-}}
\newcommand\R{\mathbb{R}}

\newcommand\UU{\mathcal{U}}
\newcommand\VV{\mathcal{V}}
\newcommand\GG{\mathcal{G}}

\newcommand\ones{\mathbbm{1}}

\newtheorem{thm}{Theorem}[section]
\newtheorem{prop}[thm]{Proposition}
\newtheorem{lem}[thm]{Lemma}
\newtheorem{cor}[thm]{Corollary}
\theoremstyle{definition}

\newtheorem{ex}[thm]{Example}

\DeclareMathOperator*{\argmax}{arg\,max}
\DeclareMathOperator*{\vertices}{vert}
\DeclareMathOperator*{\Int}{int}
\DeclareMathOperator*{\ri}{ri}

\DeclareMathOperator*{\rank}{rank}

\DeclareMathOperator*{\cl}{cl}
\DeclareMathOperator*{\epi}{epi}
\DeclareMathOperator*{\conv}{conv}
\DeclareMathOperator*{\cone}{cone}

\DeclareMathOperator*{\ran}{ran}

\DeclareMathOperator*{\br}{best\,resp}
\DeclareMathOperator*{\Nash}{Nash}
\DeclareMathOperator*{\extNash}{extNash}
\DeclareMathOperator*{\NashFaces}{NashFaces}
\DeclareMathOperator*{\maxNashFaces}{maxNashFaces}

\makeatletter
\newcommand{\leqnomode}{\tagsleft@true\let\veqno\@@leqno}
\newcommand{\reqnomode}{\tagsleft@false\let\veqno\@@eqno}
\makeatother

\title{Computing all Nash equilibria of low-rank bi-matrix games}
\author{Zachary Feinstein\thanks{Stevens Institute of Technology, School of Business, Hoboken, NJ 07030, USA, zfeinste@stevens.edu} \and Andreas L{\"o}hne\thanks{Friedrich Schiller University Jena, Faculty of Mathematics and Computer Science, 07737 Jena, Germany, andreas.loehne@uni-jena.de} \and Birgit Rudloff\thanks{Vienna University of Economics and Business, Institute for Statistics and Mathematics, Vienna A-1020, AUT, brudloff@wu.ac.at}}

\begin{document}
\maketitle

\begin{abstract} \noindent  
We study constrained bi-matrix games, with a particular focus on low-rank games. Our main contribution is a framework that reduces low-rank games to smaller, equivalent constrained games, along with a necessary and sufficient condition for when such reductions exist. Building on this framework, we present three approaches for computing the set of extremal Nash equilibria, based on vertex enumeration, polyhedral calculus, and vector linear programming. Numerical case studies demonstrate the effectiveness of the proposed reduction and solution methods.
\end{abstract}
\medskip

\section{Introduction}

Bi-matrix games were first proposed by John von Neumann for the two-player zero-sum setting \cite{neumann1928theorie}. For such games, there exists an equilibrium point that can be computed via linear programming via a minimax formulation.
A generalization of this equilibrium notion was later developed by John Nash (\cite{nash1950,nash1951}). This notion, called the Nash equilibrium, is any stable strategy set so that no player would choose to unilaterally alter his or her actions.
Within this work, we will focus on two-player bi-matrix games with constraints in a non-zero-sum setting (though we treat the zero-sum game as a special case).

In recent years, algorithms have been developed to find all extremal Nash equilibria for bi-matrix games. We refer the interested reader to, e.g., \cite{vStengel21}. 
Tools like \texttt{lrsnash}~\cite{Avis2000,Avis2006,Avis2009}, perform very well in practice and can handle relatively large game matrices. Nevertheless, the curse of dimensionality remains inherent, as the size of the strategy sets increases.

Within this work, we consider bi-matrix games that have low rank so as to reduce dimensionality. The rank of bi-matrix games were first considered in \cite{kannan2010games}. Specifically, in that work, the rank of the game is defined by the rank of the sum of the payoff matrices; in that way, e.g., a zero-sum game has 0 rank. 

In contrast to prior works, we approach the low-rank game problem by reducing the dimension of the game through a novel matrix factorization. We provide a necessary and sufficient condition for the existence of such a factorization and show that the original bi-matrix game is fully characterized by the (typically smaller) reduced game.
In doing such a reduction, we reformulate the original bi-matrix game as a constrained bi-matrix game, which can be seen as a special case of polytope games \cite{bhattacharjee2000polytope}.

However, in making this reformulation, the aforementioned methodologies no longer apply as they rely on the, typical, simplex constraints on player strategies which are replaced by polytope constraints.
Although our primary motivation comes from low-rank games, we present three methods for identifying extremal Nash equilibria in constrained bi-matrix games. We view these methods as a proof of concept rather than a fully developed framework: First, we show that vertex enumeration remains a viable approach for computing extremal Nash equilibria in constrained bi-matrix games, suggesting that tools such as \texttt{lrsnash} can likely be generalized to this setting. Second, we show that the problem can be decomposed into a few basic operations on convex polyhedra, enabling straightforward implementation using polyhedral calculus tools such as \textit{bensolve tools} \cite{bt, CirLoeWei19}, which is based on the \textit{bensolve} solver \cite{bensolve,LoeWei17} for vector linear programs (VLP).
 A third approach demonstrates that a vector linear program solver can be applied directly. We formulate two VLPs and show that the extremal Nash equilibrium points can be obtained from their solutions.


In fact, constraints on player actions arise naturally in applications such as \cite{ConstraintGames14}. 
As highlighted in \cite{MenZha14}, the constraints can be remapped into a higher-dimensional bi-matrix game with standard simplex constraints. However, this approach, in which the constraints are remapped, can be problematic, as it may lead to a dramatic increase in both the matrix size and the rank of the game.
 As such, while theoretically valid, such an approach is not practical for most cases.

The remainder of this paper is organized as follows. In Section~\ref{sec:notation}, we introduce the notation used throughout the paper. Section~\ref{sec:main} presents the fundamentals of constrained bi-matrix games. Section~\ref{sec:lowrank} focuses on low-rank games and their reduction into typically smaller constrained games. In Section~\ref{sec_methods}, we discuss three approaches for solving constrained bi-matrix games. Finally, Section~\ref{sec:numeric} presents numerical case studies to demonstrate the efficacy of our approach.

\section{Notation and preliminaries}\label{sec:notation}

A set \( P \subseteq \mathbb{R}^n \) is called a \emph{polyhedral convex set} or a \emph{convex polyhedron} if it can be described as
\begin{equation} \label{hrep0}
	P = \{ x \in \mathbb{R}^n \mid A x \leq b \},
\end{equation}
for some matrix \( A \in \mathbb{R}^{m \times n} \) and vector \( b \in \mathbb{R}^m \). This form is referred to as an \emph{H-representation} of \( P \). The \textit{recession cone} of $P$, denoted $0^+ P$, is obtained if the vector $b$ in \eqref{hrep0} is replaced by the zero vector.

Alternatively, a convex polyhedron \( P \) can be written as the Minkowski sum of the convex hull of finitely many points \( v^1, \dots, v^k \in \mathbb{R}^n \) and the conic hull of finitely many directions \( d^1, \dots, d^l \in \mathbb{R}^n \), that is,
\begin{equation} \label{vrep0}
	P = \conv\{v^1, \dots, v^k\} + \cone\{d^1, \dots, d^l\}.
\end{equation}
Such a representation is called a \emph{V-representation}. We use the convention \( \cone \emptyset = \{0\} \) to represent polytopes without any directions. 

A point \( v \in P \) is called a \emph{vertex} of \( P \) if it cannot be written as a strict convex combination of two distinct points in \( P \). The set of all vertices of \( P \) is denoted by \( \vertices P \). A nonzero direction \( d \in 0^+ P \) is called an \emph{extremal direction} if for all \( u, w \in 0^+ P \) with \( d = u + w \), it follows that \( u, w \in \cone\{d\} \). 

A \emph{P-representation} (short for \emph{projection representation}) of a convex polyhedron \( P \subseteq \mathbb{R}^n \) is an expression of the form
\begin{equation} \label{prep}
	P = \{ x \in \mathbb{R}^n \mid \exists y \in \mathbb{R}^k:\; A x + B y \leq b \},
\end{equation}
for matrices \( A \in \mathbb{R}^{m \times n} \), \( B \in \mathbb{R}^{m \times k} \), and vector \( b \in \mathbb{R}^m \). We can apply Fourier–Motzkin elimination (see, e.g., \cite{lauritzen}) to eliminate the variables \( x_1, \dots, x_n \). This yields an H-representation of \( P \), confirming that \eqref{prep} defines a convex polyhedron.
Clearly, H-representations are a special case of P-representations. Moreover, the V-representation in \eqref{vrep0} can also be viewed as a special case of a P-representation. If \( V \) is the matrix whose columns are \( v^1, \dots, v^k \) and \( D \) has columns \( d^1, \dots, d^l \), then
\[
P = \left\{ x \in \mathbb{R}^n \;\middle|\; \exists (\lambda, \mu) \in \mathbb{R}^q \times \mathbb{R}^r: V \lambda + D \mu = x,\; \lambda \geq 0,\; e^\top \lambda = 1,\; \mu \geq 0 \right\},
\]
where \( e^\top = (1, \dots, 1) \).

A \emph{proper face} of a convex polyhedron \( P \subseteq \mathbb{R}^n \) is the intersection of \( P \) with a supporting hyperplane. The empty set and $P$ itself are improper faces of $P$. A \emph{facet} of \( P \) is a face of dimension \( \dim P - 1 \), i.e., a maximal proper face.

A \textit{relative interior} point of a convex set $S$ is a point that is interior relative to the affine hull of $S$. The set of all relative interior points of $S$ is denoted by $\ri S$.

A matrix $M \in \R^{m \times n}$ is identified with a linear map $M:\R^n \to \R^m$, i.e., we have $Mx = M(x)$. For a subset $S \subseteq \R^n$, the set
$$ M[S] \coloneqq \Set{Mx |\; x \in S}$$
is called the \textit{image} of $S$ under the linear map $M$. For a set $S' \subseteq \R^m$, the set
$$ M^{-1}[S']\coloneqq \Set{x \in \R^n |\; Mx \in S'}$$
is called the \textit{inverse image} of $S'$ under the linear map $M$. The \textit{range} of a matrix $M \in \R^{m \times n}$ is the linear subspace $\ran M \coloneqq M[\R^n]$ of $\R^m$, while the \textit{kernel} of $M$ is the linear subspace $\ker M \coloneqq M^{-1}(\{0\})$ of $\R^n$.

The following result is elementary; for completeness and convenience, we include a short proof.
\begin{prop}\label{prop:im_inv}
 Let $S \subseteq \R^n$, $S' \subseteq \R^m$ and $M\in \R^{n\times m}$. Then
 $$ M^{-1}[M[S]]= S + \ker M \supseteq S,$$
 $$ M[M^{-1}[S']]= S' \cap \ran M \subseteq S'.$$
\end{prop}
\begin{proof}
We compute:
\begin{align*}
    M^{-1}[M[S]] &= \Set{x \given  M x \in M[S]} \\
                 &= \Set{x \given  \exists z \in S:\; M x = M z} \\
                 &= \Set{x \given  \exists z \in S:\; x-z \in \ker M} \\
                 &= \ker M + S.
\end{align*}
Similarly,
\begin{align*}
    M[M^{-1}[S']] &= \Set{M z \given z \in M^{-1}[S']} \\
                  &= \Set{y   \given  \exists  z \in M^{-1}[S']:\; y = M z } \\
                  &= \Set{y   \given  \exists  z:\; M z \in S', \; y = M z } \\
                  &= \Set{y   \given  \exists  z:\; y \in S',\; y = M z} \\
                  &= S' \cap \ran M
\end{align*}
completes the proof.
\end{proof}

\section{Constrained bi-matrix games} \label{sec:main}

Let $(A,B)$ be a bi-matrix game given by the payoff matrices $A, B \in \R^{m \times n}$. Player 1 chooses a row index $i \in [m]$ randomly based on a discrete probability distributions $x \in S \coloneqq \Set{x \in \R^m_+ \given \ones^{\top} x = 1}$. A vector $x \in S$ is called a {\em (mixed) strategy} of player 1. Likewise, player 2 chooses a column index $j \in [n]$ randomly with $y \in T\coloneqq \Set{y \in \R^n_+ \given \ones^{\top} y = 1}$. A vector $y \in T$ is called a {\em (mixed) strategy} of player 2. The expected payoff for player 1 is $y^{\top} A^{\top} x$. For player 2 it is $x^{\top} B y$. Each player tries to maximize her or his expected payoff. The resulting bi-linear payoff function is denoted
$$ p : \R^m \times \R^n \to \R \times \R, \qquad p(x,y) \coloneqq \begin{pmatrix}
	 y^{\top} A^{\top} x \\ x^{\top} B y
\end{pmatrix}\text{.}$$

A broader framework within bi-matrix games, termed the {\em constrained bi-matrix game}, is obtained by permitting the sets $S$ and $T$ to be arbitrary convex polytopes in $\R^m$ and $\R^n$ respectively. 

Let $\GG = (A,B;S,T)$ be a constrained bi-matrix game. For a strategy $y \in T$ of player 2, the best response of player 1 is the set of strategies
$$ \br (y) \coloneqq \argmax_{x \in S} y^{\top} A^{\top} x.$$
Likewise we define
$$ \br (x) \coloneqq \argmax_{y \in T} x^{\top} B y.$$
A tuple $(x,y) \in S \times T$ is called a {\em Nash equilibrium} of $\GG$ provided
\begin{equation}\label{eq:def_nash}
	x \in \br(y) \quad\text{and}\quad y\in \br(x).
\end{equation}
The set of all Nash equilibria of $\GG$ is denoted by $\Nash(\GG)$. 

\begin{prop}
	Every constrained bi-matrix game $\GG = (A,B;S,T)$ has at least one Nash-equilibrium if $S$ and $T$ are compact sets.
\end{prop}
\begin{proof}
	Like in the original proof by Nash \cite{Nash50} we can apply Kakutani's fixed point theorem \cite{Kakutani41} to the set-valued mapping $(x,y) \mapsto \br(y) \times \br(x)$. In particular, this requires compactness of the sets $S$ and $T$.
\end{proof}
The best response $\br (y)$ of player 1 is the set of optimal solutions of the linear program
\leqnomode
\begin{gather}\label{p1}\tag{P$_1(y)$}
  \max y^{\top} A^{\top} x \;\text{ s.t. } x \in S\text{.}
\end{gather}
\reqnomode
Likewise, the best response $\br (x)$ of player 2 is the set of optimal solutions of the linear program
\leqnomode
\begin{gather}\label{p2}\tag{P$_2(x)$}
  \max x^{\top} B y \;\text{ s.t. } y \in T\text{.}
\end{gather}
\reqnomode
We now define the following \textit{optimal value functions} for these linear programs:
$$ \eta : \R^n \to  \R \cup \Set{\infty},\; \eta(y) \coloneqq \left\{ \begin{array}{cl}
    \text{optimal value of \eqref{p1}} & \text{ if } y \in T  \\
    \infty & \text{ otherwise} 
\end{array} \right.\text{,}$$
$$ \xi : \R^m \to  \R \cup \Set{\infty},\; \xi(x) \coloneqq \left\{ \begin{array}{cl}
    \text{optimal value of \eqref{p2}} & \text{ if } x \in S  \\
    \infty & \text{ otherwise} 
\end{array} \right.\text{.}$$

\begin{prop}\label{prop:3a}
    The optimal value functions $\eta$ and $\xi$ are polyhedral convex, i.e., their epigraphs are convex polyhedra. 
\end{prop}
\begin{proof}
    Since $S$ is a polyhedron in $\R^m$, there are a matrix $\bar S \in \R^{k \times m}$ and a vector $\bar s \in \R^k$ such that
    $$S = \Set{x \in \R^m \given \bar S x \leq \bar s}.$$
Then the dual linear program of \eqref{p1} is  
\leqnomode
\begin{gather}\label{ld1}\tag{D$_1(y)$}
  \min \bar s^{\top} u \;\text{ s.t. } \bar S^{\top} u = A y, \; u \geq 0\text{.}
\end{gather}
\reqnomode
Since $S$ is assumed to be nonempty and bounded, linear programming duality yields a P-representation of the epigraph of $\eta$:
$$ \epi \eta = \Set{(y,r) \given \exists u \in \R^k: \bar S^{\top} u = A y, \; u \geq 0,\; r \geq \bar s^{\top} u}.$$
This shows that $\epi \eta$ is a convex polyhedron. The remaining part of the proof is analogous.
\end{proof}
The following characterization of Nash equilibrium points follows directly from the definitions.
\begin{prop}\label{prop:4a}
A point $(x,y) \in \R^m \times \R^n$ is a Nash equilibrium point for $\GG$ if and only if $(x,y) \in S \times T$, and
\begin{equation}\label{eq:41a}
	y^{\top} A^{\top} x = \eta(y) \quad \wedge \quad x^{\top} B y = \xi(x)\text{.}
\end{equation}
\end{prop}

The following proposition introduces a feature that is crucial for understanding the structure of the set of all Nash equilibria.

\begin{prop}\label{prop:34} 
    Let $(\bar x,\bar y) \in \Nash \GG$. Further let $F$ be a face of $\epi\xi$ with $(\bar x,\xi(\bar{x}))^{\top} \in \ri F$ and $G$ be a face of $\epi\eta$ with $(\bar y,\eta(\bar{y}))^{\top} \in \ri G$. Then, for all $(x,r)^{\top} \in F$ and all $(y,s)^{\top}\in G$ we have
    $$ r = \xi(x),\quad s = \eta(y),\quad (x,y) \in \Nash \GG.$$
\end{prop}
\begin{proof}
    Let $(x,r)^{\top} \in F$. Then the prolongation principle for relative interior points (see e.g., \cite[Proposition 1.3.3]{bertsekas}) yields a point $(\hat x,\hat r)^{\top} \in F$ such that $\bar x = \lambda x + (1-\lambda) \hat x$ and 
    $\xi(\bar x) = \lambda r + (1-\lambda) \hat r$ for some $\lambda \in (0,1)$.
Using the convexity of $\xi$, we get
$$
 \xi(\bar x) = \lambda r + (1-\lambda) \hat r \geq \lambda \xi(x) + (1-\lambda) \xi(\hat x) \geq \xi(\bar x).
$$
This implies 
$$
 \lambda r + (1-\lambda) \hat r = \lambda \xi(x) + (1-\lambda) \xi(\hat x).
$$
Since $r \geq \xi(x)$ and $\hat r \geq \xi(\hat x)$, we conclude $r = \xi(x)$.
    
    A similar statement holds for $(y,s)^{\top} \in G$ and we have
    $\bar y = \mu y + (1-\mu) \hat y$ and $\eta(\bar y) = \mu s + (1-\mu) \hat s$ for some $\mu \in (0,1)$, and $s = \eta(y)$.
  
    By the definition of $\xi$ we get the inequalities
    $x^{\top} B y \leq \xi(x)$,
    $\hat x^{\top} B y \leq \xi(\hat x)$,
    $x^{\top} B \hat y \leq \xi(x)$, and
    $\hat x^{\top} B \hat y \leq \xi(\hat x)$.
    Multiplying the inequalities, respectively, by $\lambda \mu$, $(1-\lambda)\mu$, $\lambda(1-\mu)$, $(1-\lambda)(1-\mu)$ and taking the sum yields $\bar x^{\top} B \bar y = \xi(\bar x)$, where the equality follows from Proposition \ref{prop:4a} and the assumption that $(\bar x, \bar y) \in \Nash \GG$. Therefore, the initial four inequalities must hold with equality, too. In particular, we obtain $x^{\top} B y = \xi(x)$. Analogously, we can show that
    $y^{\top} A^{\top} x = \eta(y)$. Now the claim follows from Proposition \ref{prop:4a}.
\end{proof}

As a consequence of the preceding proposition, we see that the set of all Nash equilibria is the union of (some of the) finitely many sets of the form \( \tilde F_i \times \tilde G_j \), where $\tilde F_i \coloneqq \Set{x\, |\; \exists r \in \R: (x,r)^{\top} \in F_i}$, $\tilde G_j \coloneqq \Set{y\, |\; \exists s \in \R: (y,s)^{\top} \in G_j}$,  \( F_i \) is a proper face of \( \operatorname{epi} \eta \), and \( G_j \) is a proper face of \( \operatorname{epi} \xi \). These sets are closely related to the \textit{Nash sets} introduced in \cite{millman74}.


A Nash equilibrium point $(x,y)$ of a constrained bi-matrix game $\GG$ is called {\em extremal} if the expression $(x,y,p(x,y))$ cannot be expressed as a proper convex combination of $(x^i,y^i,p(x^i,y^i))$ for finitely many Nash equilibrium points $(x^i,y^i)$, $i \in [k]$. The set of all extremal Nash equilibria is denoted by $\extNash(\GG)$. 
Our aim is to determine this set for any given bi-matrix game $\GG$. We start with a characterization in terms of vertices of the epigraphs of $\eta$ and $\xi$.
\begin{prop}\label{prop:6a}
    The following is equivalent:
    \begin{enumerate}[(i)]
        \item $(x,y) \in \extNash \GG$,
        \item $\begin{pmatrix}
       y \\ y^{\top} A^{\top} x
   \end{pmatrix} \in \vertices (\epi \eta)$ and $\begin{pmatrix}
       x \\ x^{\top} B y
   \end{pmatrix} \in \vertices (\epi \xi)$.
    \end{enumerate}
\end{prop}
\begin{proof}
   (i) $\Rightarrow$ (ii). Let $(x, y) \in \extNash \GG$. 
   By Proposition \ref{prop:4a}, we have $x^{\top} B y = \xi(x)$. Using, for instance, \cite[Theorem 18.2]{Rockafellar70}
   we see that, the point 
   $(x,x^{\top} B y)$ belongs to the relative interior of a proper face $F$ of $\epi \xi$. Without loss of generality, let the second part of (ii) be violated. Then $\dim F > 0$. Thus there are $x^1, x^2$ such that $(x,\xi(x))$ is a proper convex combination of $(x^1,\xi(x^1))$ and $(x^2,\xi(x^2))$. Consequently, $(x,y,p(x,y))$ is a proper convex combination of $(x^1,y,p(x^1,y))$ and $(x^2,y,p(x^2,y))$. By Proposition \ref{prop:34}, $(x^1,y),(x^2,y) \in \Nash \GG$, which contradicts (i).  

   (ii) $\Rightarrow$ (i).
   Assume (i) is violated while (ii) holds. Then there is a proper convex combination 
   $$ (x,y,p(x,y)) = \sum_{i=1}^\ell \lambda_i (x^i,y^i,p(x^i,y^i))$$
   where $(x^i,y^i) \in \Nash \GG$ for all $i \in [\ell]$.
   Using Proposition \ref{prop:4a} we obtain 
   \begin{equation}\label{eq:convcomb}
        x = \sum_{i=1}^\ell \lambda_i x^i \quad \text{ and } \quad
   y = \sum_{i=1}^\ell \lambda_i y^i
   \end{equation}
   $$ \xi(x) = \sum_{i=1}^\ell \lambda_i \xi(x^i) \quad \text{ and } \quad
    \eta(y) = \sum_{i=1}^\ell \lambda_i \eta(y^i)\text{,}$$
where at least one of the convex combinations in \eqref{eq:convcomb} is proper. This contradicts statement (ii).    
\end{proof}

\begin{prop}
    Every $(x,y) \in \Nash \GG$ can be expressed as a convex combination of extremal Nash equilibria.
\end{prop}
\begin{proof}
    The point $(x,\xi(x))$ belongs to some proper face $F$ of $\epi \xi$. Thus, it can be expressed as a convex combination of the vertices $(x^1,\xi(x^1)),\dots,(x^k,\xi(x^k))$ of $F$, that is, 
    $$ (x,\xi(x)) = \sum_{1=1}^k \lambda_i (x^i,\xi(x^i))$$
    for coefficients $\lambda_1,\dots,\lambda_k \geq 0$ with $\lambda_1 +\dots + \lambda_k = 1.$
    Likewise, there are vertices $(y^1,\eta(y^1)),\dots, (y^\ell,\eta(y^\ell))$ of a face $G$ of $\epi \eta$ such that $(y,\eta(y))$ is a convex combination of these points, where the coefficients are denoted by $\mu_1,\dots,\mu_\ell$ here. The face $F$ can be chosen such that $(x,\xi(x))$ belongs to the relative interior of $F$, and likewise for the face $G$. Propositions \ref{prop:34} and \ref{prop:6a} yield that $(x^i,y^j) \in \extNash \GG$ for all $(i,j) \in [k] \times [\ell]$. Clearly, $(x,y)$ can be written as a convex combination of the $(x^i,y^j)$, where the coefficients are $\gamma_{ij} \coloneqq \lambda_i\mu_j$.    
\end{proof}

These considerations lead to the following method which computes the set $\extNash{\GG}$ of a constrained bi-matrix game $\GG=(A,B;S,T)$. 
\begin{enumerate}[(A)]
	\item Compute all vertices $(x^1,\alpha_1), \dots, (x^k,\alpha_k)$ of $\epi \xi$.
 \item Compute all vertices $(y^1,\beta_1), \dots, (y^\ell,\beta_\ell)$ of $\epi \eta$.
	\item Iterate over $(i,j) \in [k] \times [\ell]$. If
	$$ (x^i)^{\top} B y^j = \alpha_i \quad \text{ and } \quad  (y^j)^{\top} A^{\top} x^i = \beta_j\text{,}$$
	mark $(x^i,y^j)$ as extremal Nash equilibrium point of $\GG$. 
\end{enumerate}

\section{Low rank games}\label{sec:lowrank}

In this section, we present a method for replacing low-rank bi-matrix games with typically smaller constrained bi-matrix games, referred to as the \textit{reduced game}, which are equivalent in a precisely defined sense. We also provide a necessary and sufficient condition for the validity of this reduction.


Given a game \( \mathcal{G} = (A, B, S, T) \), we consider a singular value decomposition of \( A + tB \), where \( t \in \mathbb{R} \setminus \{0\} \) is a fixed arbitrary nonzero real number. Let \( k \) denote the rank of \( A + tB \). Then, there exist matrices \( U \in \mathbb{R}^{m \times k} \), \( V \in \mathbb{R}^{n \times k} \), and \( \Sigma \in \mathbb{R}^{k \times k} \) such that 
\begin{equation}\label{eq:svd}
   A + tB = U \Sigma V^{\top}, \quad U^{\top} U = I_k, \quad V^{\top} V = I_k, 
\end{equation}
where \( I_k \) is the \( k \times k \) identity matrix, and \( \Sigma \) is a diagonal matrix containing the \( k \) nonzero singular values of \( A + tB \).
Note that \( U U^{\top} \neq I_m \) and \( V V^{\top} \neq I_n \) in general.

The game $ \bar \GG = ( \bar A, \bar B, \bar S, \bar T )$ with 
$$ \bar A \coloneqq U^{\top} A V \in \R^{k \times k}, \quad \bar B \coloneqq U^{\top} B V \in \R^{k \times k},$$
$$ \bar S \coloneqq U^{\top}[S], \quad \bar T \coloneqq V^{\top}[T],$$
is called the \textit{reduced game} of $\GG$. Likewise, we refer to the reduced payoff matrices $\bar{A}, \bar{B}$, and the reduced constraint sets $\bar{S}, \bar{T}$.

It is desirable to be able to reconstruct the original game from the reduced one. In particular, we aim to satisfy the condition
\begin{equation}\label{eq:restore_payoff}  
A = U \bar{A} V^{\top} \quad \text{and} \quad B = U \bar{B} V^{\top},  
\end{equation}
which is motivated by the following result.

\begin{prop}\label{prop:56}
    Let $\bar \GG  = (\bar A, \bar B, \bar S, \bar T)$ be a reduced game of $\GG = (A, B, S, T)$ such that \eqref{eq:restore_payoff} is satisfied. Then, $(x,y) \in \Nash \GG$ if and only if 
    $$ (U^{\top} x,\, V^{\top} y) \in \Nash \bar{\GG} \quad\wedge\quad (x,y) \in S \times T.$$
\end{prop}
\begin{proof}
    By definition, \((x, y) \in \Nash \GG\) if and only if 
    \((x, y) \in S \times T\) and 
    \[
    x \in \argmax_{y \in T} y^{\top} A^{\top} x, \qquad y \in \argmax_{x \in S} x^{\top} B y.
    \]
    If \eqref{eq:restore_payoff} holds, this is equivalent to 
    \[
    V^{\top} x \in \argmax_{U^{\top} y \in U^{\top}[T]} \underbrace{y^{\top} U}_{= \bar y^{\top}} \bar A^{\top} \underbrace{V^{\top} x}_{= \bar x}, \qquad U^{\top} y \in \argmax_{V^{\top} x \in V^{\top}[S]} \underbrace{x^{\top} V}_{= \bar x^{\top}} \bar B \underbrace{U^{\top} y}_{= \bar y},
    \]
    which implies the result.
\end{proof}

Whether the original payoff matrices can be recovered from the reduced payoff matrices according to \eqref{eq:restore_payoff} is uncertain as shown by the following example where \eqref{eq:restore_payoff} is violated:\:
\[ 
A = 
\begin{pmatrix}
0 & 1 \\
0 & 1
\end{pmatrix}, \;
B = 
\begin{pmatrix}
0 & 1 \\
0 & 0
\end{pmatrix},\]
and for $t=-1$ we get
\[
U = 
\begin{pmatrix}
0 \\
1
\end{pmatrix}, \;
V = 
\begin{pmatrix}
0 \\
1
\end{pmatrix}, \;
\bar{A} = 
\begin{pmatrix}
1 
\end{pmatrix}, \;
\bar{B} = 
\begin{pmatrix}
0
\end{pmatrix}.
\]
The following proposition proves that the restorability property of the payoff matrices in \eqref{eq:restore_payoff} is equivalent to the condition
\begin{gather}
\label{ass:A}
\begin{aligned}
    \dim(\ran A + \ran B) = \rank(A+tB) = \dim(\ran A^{\top} + \ran B^{\top})
\end{aligned}
\end{gather}
recalling that $\ran M = M[\R^n]$ is the range space of $M \in \R^{m \times n}$.
\begin{prop}\label{prop:51}
    Let $\GG = (A, B, S, T)$ be a game. The payoff matrices $A$ and $B$ of $\GG$ can be restored from the reduced game $\bar{\GG}$ according to \eqref{eq:restore_payoff} if and only if \eqref{ass:A} holds.
\end{prop}
\begin{proof}
For arbitrary matrices \( A, B \) of the same dimension and arbitrary \( t \in \mathbb{R} \setminus \{0\} \), we have
\begin{equation}\label{eq:range0}
\ran (A + t B) \subseteq \ran A + \ran B.
\end{equation}
By the first equation in \eqref{ass:A}, the subspaces on both sides of the inclusion have the same dimension. Therefore, this part of \eqref{ass:A} is equivalent to 
\begin{equation}\label{eq:range1}
\ran (A + t B) = \ran A + \ran B.
\end{equation}
Taking into account that \( \text{rank}(A + t B) = \text{rank}(A^{\top} + t B^{\top}) \), we see that the second equation in \eqref{ass:A} is equivalent to 
\begin{equation}\label{eq:range2}
\ran (A^{\top} + t B^{\top}) = \ran A^{\top} + \ran B^{\top}.
\end{equation}

The columns of \( A \) belong to \( \ran A \), and by \eqref{eq:range1}, they also belong to \( \ran (A + t B) \). Hence, there exists a matrix \( X \in \mathbb{R}^{n \times n} \) such that \( A = (A + t B) X \). Using the singular value decomposition of \( A + t B \) in \eqref{eq:svd}, we obtain \( A = U M \) for \( M \coloneqq \Sigma V^{\top} X \). This implies \( U U^{\top} A = U M = A \). Similarly, we obtain \( U U^{\top} B = B \). Using \eqref{eq:range2} instead of \eqref{eq:range1}, we obtain \( A = A V V^{\top} \) and \( B = B V V^{\top} \). We conclude
\[
A = U U^{\top} A V V^{\top} = U \bar{A} V^{\top},
\]
and a corresponding statement for \( B \), i.e., \eqref{eq:restore_payoff} holds.

Now, assume \eqref{eq:restore_payoff} is satisfied. Let \( y \in \ran A + \ran B \), i.e., there exist vectors \( x \) and \( z \) such that
\[
y = A x + B z = U (\bar{A} V^{\top} x + \bar{B} V^{\top} z).
\]
Then \( y \in \ran U \), i.e., \( \ran A + \ran B \subseteq \ran U \) holds. In \eqref{eq:svd}, we have \( \text{rank}(\Sigma V^{\top}) = k \), so every vector \( z \in \mathbb{R}^k \) can be expressed as \( z = \Sigma V^{\top} x \) for some \( x \in \mathbb{R}^n \). Thus, \( \ran U = \ran (A + t B) \), and therefore,
\[
\ran A + \ran B \subseteq \ran (A + t B).
\]
Using \eqref{eq:range0}, we obtain \eqref{eq:range1}. By similar arguments, we obtain \eqref{eq:range2}.
\end{proof}

Note that the assumption
\begin{equation}\label{eq:rankcond}
    \rank(A + t B) = \rank A + \rank B
\end{equation}
implies condition \eqref{ass:A}. Indeed we have 
\begin{align*}
\rank A + \rank B &= \dim (\ran A) + \dim (\ran B) \\
&\geq \dim (\ran A + \ran B) \\
&\geq \dim (\ran (A + t B)) \\
&= \rank (A + t B),
\end{align*}
where the last inequality follows from \eqref{eq:range0}. Thus by \eqref{eq:rankcond} we obtain the first equation of \eqref{ass:A}. The second one follows analogously taking into account that the row rank of a matrix equals the column rank.

However \eqref{eq:rankcond} is not equivalent to \eqref{ass:A} and thus not necessary for the restorability of the payoff matrices in \eqref{eq:restore_payoff}, as demonstrated by the following example:
\[
A = \begin{pmatrix}
0 & 1 & 0 \\
0 & 0 & 0 \\
0 & 1 & 1
\end{pmatrix},
\quad
B = \begin{pmatrix}
0 & 0 & 0 \\
0 & 0 & 0 \\
0 & 1 & 0
\end{pmatrix}.
\]
We have $\rank A = 2$, $\rank B = 1$, and $\rank (A+ tB)=2$ for all $t \in \R \setminus\Set{0}$, thus \eqref{eq:rankcond} is violated. Choosing, for example, $t=-1$, we obtain: 
\[
U =
\begin{pmatrix}
1 & 0 \\
0 & 0 \\
0 & 1
\end{pmatrix},
\;
V =
\begin{pmatrix}
0 & 0 \\
1 & 0 \\
0 & 1
\end{pmatrix},
\quad
\bar{A} =
\begin{pmatrix}
1 & 0 \\
1 & 1
\end{pmatrix},
\;
\bar{B} =
\begin{pmatrix}
0 & 0 \\
1 & 0
\end{pmatrix}.
\]
We see that \eqref{eq:restore_payoff}, and equivalently \eqref{ass:A}, are satisfied.

For the reduced game $\bar \GG$ we consider the linear programs
\leqnomode
\begin{gather}\label{p1bar}\tag{$\bar{\rm P}_1(y)$}
  \max y^{\top} \bar{A}^{\top} x \;\text{ s.t. } x \in \bar{S}\text{,}
\end{gather}
\reqnomode
\leqnomode
\begin{gather}\label{p2bar}\tag{$\bar{\rm P}_2(x)$}
  \max x^{\top} \bar{B} y \;\text{ s.t. } y \in \bar{T}\text{,}
\end{gather}
\reqnomode
and the respective optimal value functions
$$ \bar{\eta} : \R^k \to  \R \cup \Set{\infty},\; \bar{\eta}(y) \coloneqq \left\{ \begin{array}{cl}
    \text{optimal value of \eqref{p1bar}} & \text{ if } y \in \bar{T}  \\
    \infty & \text{ otherwise} 
\end{array} \right.\text{,}$$
$$ \bar{\xi} : \R^k \to  \R \cup \Set{\infty},\; \bar{\xi}(x) \coloneqq \left\{ \begin{array}{cl}
    \text{optimal value of \eqref{p2bar}} & \text{ if } x \in \bar{S}  \\
    \infty & \text{ otherwise} 
\end{array} \right.\text{.}$$
If \eqref{eq:restore_payoff} holds—or equivalently, \eqref{ass:A}—then the optimal value functions of $\GG$ and $\bar{\GG}$ are related, as shown in the next two propositions.

\begin{prop}\label{prop:52a}
    Let $\bar \GG  = (\bar A, \bar B, \bar S, \bar T)$ be a reduced game of $\GG = (A, B, S, T)$ such that \eqref{eq:restore_payoff} is satisfied. Then
$$\forall y \in T:\; \eta(y) = \bar{\eta}(V^{\top} y),$$
$$\forall x \in S:\; \xi(x) = \bar{\xi}(U^{\top} x).$$
\end{prop}
\begin{proof}
    Let $y \in T$. Using \eqref{eq:restore_payoff} we obtain
    \begin{align*}
         \eta(y) & = \max_{x \in S} y^{\top} A^{\top} x \\
                 & = \max_{x \in S} y^{\top} \bar{V} \bar{A}^{\top} \bar{U}^{\top} x \\
                 & = \max_{\bar x \in \bar S} y^{\top} \bar{V} \bar{A}^{\top} \bar{x} 
                 = \bar{\eta}(\bar{V}^{\top} y).
    \end{align*}
    The proof of the second statement is analogous. 
\end{proof}

\begin{prop}\label{prop:52}
    Let $\bar \GG  = (\bar A, \bar B, \bar S, \bar T)$ be a reduced game of $\GG = (A, B, S, T)$ such that \eqref{eq:restore_payoff} is satisfied. Then
    $$ \begin{pmatrix} y \\ r \end{pmatrix} \in \epi \eta \quad \iff \quad 
    \left[ \begin{pmatrix} \bar{V} y \\ r \end{pmatrix} \in \epi \bar{\eta} \;\;\wedge\;\; y \in T \right ],$$
    $$ \begin{pmatrix} x \\ r \end{pmatrix} \in \epi \xi \quad \iff \quad \left[ \begin{pmatrix} \bar{U} x \\ r \end{pmatrix} \in \epi \bar{\xi} \;\;\wedge\;\; x \in S \right].$$    
\end{prop}
\begin{proof}
Using Proposition \ref{prop:52a}, we get
    \begin{align*}
        \begin{pmatrix} y \\ r \end{pmatrix} \in \epi \eta \quad &\iff \quad r \geq \eta(y) = \bar{\eta}(\bar{V}^{\top} y), \; y \in T \\
        &\iff \quad  \begin{pmatrix} \bar{V} y \\ r \end{pmatrix} \in \epi \bar{\eta}, \; y \in T 
    \end{align*}
    The proof of the second statement is analogous.
\end{proof}

As a consequence of Proposition \ref{prop:52} we obtain the following representations of epigraphs. 
We use the notation 
$$ \VV \coloneqq  \begin{pmatrix} 
\bar{V}^{\top} & 0 \\ 
0 & 1 
\end{pmatrix}, \qquad  \UU \coloneqq \begin{pmatrix} 
\bar{U}^{\top} & 0 \\ 
0 & 1 
\end{pmatrix} .$$
\begin{cor}\label{cor:53}
    Let $\bar \GG  = (\bar A, \bar B, \bar S, \bar T)$ be a reduced game of $\GG = (A, B, S, T)$ such that \eqref{eq:restore_payoff} is satisfied. Then
   \[
\epi \bar{\eta} = 
\VV 
[\epi \eta],\qquad \epi \eta = \VV^{-1}[\epi \bar{\eta}] \cap (T \times \R),
\]
\[
\epi \bar{\xi} = 
\UU
[\epi \xi], \qquad \epi \xi = \UU^{-1}[\epi \bar{\xi}] \cap (S \times \R),\]   
\end{cor}

The following two lemmas address some technical requirements needed for the forthcoming main result.

\begin{lem}\label{lem:55}
Let $\bar \GG  = (\bar A, \bar B, \bar S, \bar T)$ be a reduced game of $\GG = (A, B, S, T)$ such that \eqref{eq:restore_payoff} is satisfied.
Let $\bar{F}$ be a face of $\epi{\bar{\xi}}$. Then
\[
F \coloneqq \UU^{-1}[\bar{F}] \cap (S \times \R)
\]
is a face of $\epi \xi$, and $\UU[F] = \bar F$. Likewise, if 
$\bar{G}$ is a face of $\epi{\bar{\eta}}$, then
\[
G \coloneqq \VV^{-1}[\bar{G}] \cap (T \times \R)
\]
is a face of $\epi \eta$, and $\VV[G] = \bar G$.
\end{lem}

\begin{proof} We first note that, by Corollary \ref{cor:53}, $F$ is a subset of $\epi \xi$. Moreover, $F$ is convex.
Now, let $z^1, z^2 \in \epi \xi$, $\lambda \in (0,1)$ such that 
$ \lambda z^1 +
(1-\lambda) z^2 \in F
$.
By Corollary \ref{cor:53} we have
$
\bar{z}^1
\coloneqq
\UU 
z^1 \in \epi \bar{\xi}$ and
$
\bar{z}^2
\coloneqq
\UU 
z^2 \in \epi \bar{\xi}
$.
Using the definition of $F$ we obtain
$
\lambda 
\bar{z}^1 +
(1-\lambda) 
\bar{z}^2 \in \bar{F}
$.  
Since $\bar F$ is a face of $\epi \bar{\xi}$, we conclude that $\bar{z}^1, \bar{z}^2\in \bar{F}$.  
Using the fact that $\epi \xi \subseteq S \times \R$, we get $z^1,z^2 \in F$. This proves that $F$ is a face of $\epi \xi$.

Further we have 
\begin{align*}
  \UU[F] & = \UU[ \UU^{-1}[\bar{F}] \cap (S \times \R)] \\ 
         & = \Set{ t \given t = \UU z,\; t \in \bar F,\; z \in S \times \R} \\
         & = \bar F \cap \UU[S \times \R]\\
         &= \bar F,
\end{align*}
where the last equation can be derived using the following consequence of Corollary \ref{cor:53}:
$$ \bar F \subseteq \epi \bar \xi = \UU[\epi \xi] \subseteq \UU[S \times \R].$$
The remaining statements can be shown similarly.
\end{proof}

\begin{lem}\label{lem:57}
    Let $P,Q \subseteq \R^n$ be nonempty convex polyhedra such that $Q \subseteq P$. Then there exists a face $F$ of $P$ such that
    $$ F \supseteq Q \quad \text{ and }  \quad \ri F \cap \ri Q \neq \emptyset.$$
\end{lem}
\begin{proof} Since $Q$ is nonempty there is some $\bar{x} \in \ri Q$. The convex polyhedron $P$ can be represented by finitely many affine inequalities:
$$P = \Set{x \in \R^n |\; (a^1)^{\top} x \leq b_1,\,\dots,\, (a^m)^{\top} x \leq b_m}.$$ 
Without loss of generality we can assume that there is some integer $k$ with $0 \leq k \leq m$ such that
$$(a^1)^{\top} \bar{x} = b_1,\,\dots,\, (a^k)^{\top} \bar{x} = b_k,$$ 
$$(a^{k+1})^{\top} \bar{x} < b_{k+1},\,\dots,\, (a^m)^{\top} \bar{x} < b_m.$$
Using $k$, which depends on $\bar{x}$, we define the affine subspace 
$$H \coloneqq \Set{x \in \R^n |\; (a^1)^{\top} x = b_1,\,\dots,\, (a^k)^{\top} x = b_k}.$$
Then $F \coloneqq H \cap P$ is a face of $P$ with $\bar x \in \ri F \cap \ri Q$.

Let $x \in Q \subseteq P$. Since $\bar x \in \ri Q$, there exists some $z \in Q \subseteq P$ and some $\lambda \in (0,1)$ such that 
$\bar x = \lambda x + (1-\lambda) z$ (prolongation principle). Using $\bar{x} \in F$ and the fact that $F$ is a face of $P$, we obtain $x \in F$. Thus we have $F \supseteq Q$.
\end{proof}

A face $F$ of the epigraph $\epi f$ of a extended real-valued function $f$ is called \textit{non-vertical} if every point of $F$ is of the form $(x,f(x))$. Two polyhedral convex sets $F \subseteq \R^{m+1}$, $G \subseteq \R^{n+1}$ are called a \textit{pair of Nash faces} of a bi-matrix game $\GG$, denoted
$$ (F,G) \in \NashFaces \GG,$$
if $F$ is a non-vertical face of $\epi \xi$, $G$ is a non-vertical face of $\epi \eta$, and
$$ \Set{x \given (x,\xi(x))^{\top} \in F} \times  \Set{y \given (y,\eta(y))^{\top} \in G} \subseteq \Nash \GG.$$
A pair $(\bar F,\bar G)\in \NashFaces \GG$ is called \textit{maximal}, denoted as
$$ (\bar F,\bar G)\in \maxNashFaces \GG,$$
if it satisfies the following property:
$$ \left[(F,G)\in \NashFaces,\; F\supseteq \bar F,\; G \supseteq \bar G\right] \implies (\bar F,\bar G)= (F,G).$$
\begin{thm}\label{thm:main}
Let $\bar \GG  = (\bar A, \bar B, \bar S, \bar T)$ be a reduced game of $\GG = (A, B, S, T)$ such that \eqref{eq:restore_payoff} is satisfied.
The mapping
$$ (F,G) \mapsto (\UU[F],\VV[G])$$ is a one-to-one correspondence between $\maxNashFaces \GG$ and $\maxNashFaces \bar{\GG}$. The inverse map is
$$ (\bar{F}\,,\,\bar{G}) \mapsto (\UU^{-1}[\bar{F}] \cap (S \times \R)\,,\; \VV^{-1}[\bar G] \cap (T \times \R)).$$
\end{thm}
\begin{proof}
Denote the mapping by $\Phi$, i.e., 
$$ \Phi((F,G)) = (\UU[F],\VV[G]). $$
We first show that for all $(F,G) \in \maxNashFaces \GG$, we have $(\UU[F],\VV[G]) \in \maxNashFaces \bar{\GG}$.
By Corollary \ref{cor:53}, $\UU[F] \subseteq \epi\bar{\xi}$ and $\VV[G] \subseteq \epi\bar{\eta}$. Since $(F,G) \in \NashFaces \GG$, any point $(f,g) \in F \times G$ is of the form 
$$f = \begin{pmatrix}
    x \\ \xi(x)
\end{pmatrix},\;  g = \begin{pmatrix}
    y \\ \eta(y)
\end{pmatrix},\;  (x,y) \in \Nash \GG.$$ 
By Propositions \ref{prop:52a} and \ref{prop:56}, we obtain that $(\bar{f},\bar{g}) \in \UU[F] \times \VV[G]$ implies
$$\bar{f} = \begin{pmatrix}
    V^{\top} x \\ \xi(V^{\top} x)
\end{pmatrix},\;  \bar{g} = \begin{pmatrix}
    U^{\top} y \\ \eta(U^{\top} y)
\end{pmatrix},\;  (V^{\top} x,U^{\top} y) \in \Nash \bar{\GG}.$$ 
By Lemma \ref{lem:57}, there are faces $\bar{F}$ of $\epi\bar{\xi}$ and $\bar{G}$ of $\epi\bar{\eta}$ such that 
$$\bar{F} \supseteq \UU[F],\; \bar{G}\supseteq Q[G],\; \ri\bar{F}\cap \ri\UU[F] \neq \emptyset,\; \ri\bar{G}\cap \ri\VV[G] \neq \emptyset.$$
By Proposition \ref{prop:34}, we have $(\bar{F},\bar{G})\in \NashFaces\bar{\GG}$. These Nash faces can be enlarged to maximal Nash faces: There are $(\hat{F},\hat{G})\in \maxNashFaces \bar{\GG}$ such that 
$$ \hat{F} \supseteq \bar{F} \supseteq \UU[F], \qquad \hat{G} \supseteq \bar{G} \supseteq \VV[G].$$
Now we set 
$$\tilde{F} \coloneqq \UU^{-1}[\hat{F}]  \cap (S \times \R), \quad \tilde{G} \coloneqq \VV^{-1}[\hat{G}]  \cap (T \times \R).$$
By Lemma \ref{lem:55}, $\tilde{F}$ is a face of $\epi \xi$ and $\tilde{G}$ is a face of $\epi\eta$.
Using Proposition \ref{prop:im_inv}, we get
$\UU^{-1}[\hat{F}] \supseteq \UU^{-1}[\UU[F]] \supseteq F$.
Since $F \subseteq \epi\xi \subseteq S \times \R$, and using a similar argumentation $\tilde{G}$, we obtain 
$$ \tilde{F}\supseteq F,\quad \tilde{G}\supseteq G.$$
Propositions \ref{prop:52a} and \ref{prop:56} imply that $(\tilde{F},\tilde{G})\in \NashFaces$. Maximality of $(F,G)$ yields $F = \tilde{F}$ and $\tilde{G}= G$. Thus we have
$F = \UU^{-1}[\hat{F}]\cap (S \times \R)$, and by Lemma \ref{lem:55}, $\UU[F] = \hat{F}$, and likewise $\VV[G]= \hat{G}$. This proves our first claim that $(\UU[F],\VV[G]) \in \maxNashFaces \bar{\GG}$.

In this second part of the proof, let $(\bar{F},\bar{G})\in \maxNashFaces\bar{\GG}$ be arbitrarily given. We show that $(F,G) \in \maxNashFaces\GG$ for
$$ F \coloneqq \UU^{-1}[\bar{F}]\cap(S\times\R),\quad G \coloneqq \VV^{-1}[\bar{G}]\cap(T\times\R).$$
 By Lemma \ref{lem:55}, $F$ is a face of $\epi \xi$ satisfying $\UU[F]=\bar{F}$, and likewise for $G$. Elements of $\bar F$ are of the form $(\bar{x},\bar{\xi}(\bar{x}))$. By Proposition \ref{prop:52a} we see that elements of $F$ must be of the form $(x,\xi(x))$ for $V^{\top} x = \bar{x}$. Using similar statements for $G$, Proposition \ref{prop:56} yields that $(F,G)\in \NashFaces\GG$. Let $(\tilde{F},\tilde{G}) \in \maxNashFaces\GG$ such that $\tilde{F}\supseteq F$ and $\tilde{G}\supseteq G$. As shown in the first part of the proof, we have $(\UU[\tilde{F}],\VV[\tilde{G}]) \in \maxNashFaces \bar{\GG}$ and 
 $$ \UU[\tilde{F}] \supseteq \UU[F] = \bar F,\quad \VV[\tilde{G}] \supseteq \VV[G] = \bar G. $$
 The maximality of $(\bar{F},\bar{G})$ implies $\UU[\tilde{F}]= \bar F$ and $\VV[\tilde{G}]= \bar{G}$.
Using Proposition \ref{prop:im_inv}, we get
$\UU^{-1}[\bar{F}] = \UU^{-1}[\UU[\tilde{F}]] \supseteq \tilde{F}$. Since $\tilde{F} \subseteq S\times \R$, we get
$$ \tilde{F} = \tilde{F} \cap (S \times \R) \subseteq \UU^{-1}[\bar{F}] \cap (S\times \R) = F.$$
This results in $\tilde{F} = F$ and likewise we obtain $\tilde{G}= G$, which proves that $(F,G)\in \maxNashFaces\GG$.

In the third part we show that 
$$ \tilde{F} \coloneqq \UU^{-1}[\UU[F]] \cap (S \times \R) = F,\quad
   \tilde{G} \coloneqq \VV^{-1}[\VV[G]] \cap (T \times \R) = G$$
holds for every $(F,G)\in \maxNashFaces \GG$.
Using Proposition \ref{prop:im_inv}, we get $\tilde{F} \supseteq F$ and $\tilde{G}\supseteq G$. As shown in the first two parts of the proof, 
$$(\tilde{F},\tilde{G})\in \maxNashFaces \GG.$$ The maximality of $(F,G)$ yields the claim.

It remains to show that 
$$ \hat{F} \coloneqq \UU[\UU^{-1}[\bar{F}] \cap (S \times \R)]= \bar{F},\quad 
   \hat{G} \coloneqq \VV[\VV^{-1}[\bar{G}] \cap (T \times \R)] = \bar{G},$$
holds for every $(\bar{F},\bar{G})\in \maxNashFaces \bar{\GG}$, which follows in a similar way.
\end{proof}

\section{Computing the vertices of epigraphs} \label{sec_methods}

In this section, we present three methods for computing the vertices of the epigraphs of the optimal value functions \( \eta \) and \( \xi \).

The first method relies on vertex enumeration. It is a classical result that the set of extremal Nash equilibrium points in unconstrained bimatrix games can be determined via vertex enumeration; that is, by computing the vertices of certain polytopes defined by affine inequalities (see~\cite{mangasarian64}).

The second method builds on the observation that a P-representation of \( \epi \eta \) was derived in the proof of Proposition~\ref{prop:3a}. The vertices of a P-represented polyhedron can be obtained by solving a \emph{polyhedral projection problem}, which is equivalent to a vector linear program—that is, a multi-objective linear program with an arbitrary polyhedral ordering cone~\cite{LoeWei16}.

Given the equivalence between vector linear programming and polyhedral projection, it is natural that a vector linear programming solver can also be applied directly. This observation forms the basis of the third method.

\bigskip
\noindent
\textbf{Via vertex enumeration:}
A point \( (y, r) \) belongs to \( \epi \eta \) if and only if
\[
y \in T \quad \text{and} \quad \forall x \in S: \; r \geq y^{\top} A^{\top} x\text{.}
\]
Since \( S \) is a polytope, it suffices to consider only the vertices of \( S \). Moreover, \( T \) can be described by finitely many linear inequalities. As a result, \( \epi \eta \) admits a representation in terms of finitely many linear inequalities, and vertex enumeration can then be used to compute its vertices. The procedure for computing the vertices of \( \epi \xi \) is analogous.

A disadvantage of this simple method is that the vertices of \( S \) and the inequalities defining \( T \) need to be computed. While this is straightforward in the case of unconstrained games, it can become computationally expensive in the constrained setting. 

In the important special case of a reduced game arising from an unconstrained bi-matrix game, the number of vertices of the reduced strategy set \( \bar{S} \) cannot exceed the number of vertices of the probability simplex \( S \); that is, \( \bar{S} \) has at most \( m+1 \) vertices. Moreover, it is not necessary to explicitly compute the vertices of \( \bar{S} \), as the finitely many inequalities can be generated by iterating over all \( \bar{x} = V^{\top} x \), where \( x \) is a vertex of the probability simplex \( S \).

\bigskip
\noindent
\textbf{Via polyhedral calculus and polyhedral projection:} Polyhedral calculus, introduced in \cite{CirLoeWei19} and implemented in the software {\it bensolve tools} \cite{CirLoeWei19,bt}, is a technique for performing a range of basic operations on polyhedral convex sets using P-representations. These operations include: 
\begin{itemize} 
\item finite intersections: $\bigcap_{i=1}^n P_i$ 
\item finite Cartesian products: $P_1 \times \dots \times P_n$ 
\item the closed conic hull: $\cl \cone P \coloneqq \cl (\R_+ P)$ 
\item the polar cone: $P^* \coloneqq \Set{y \given \forall x \in P: y^{\top} x \leq 0}$ 
\item the polar set: $P^\circ \coloneqq \Set{y \given \forall x \in P: y^{\top} x \leq 1}$
\item the image under a linear transformation $M$: 
\[
P[M] \coloneqq \Set{M x \given x \in P}
\]
\item the inverse image under a linear transformation $M$: 
\[
Q^{-1}[M] \coloneqq \Set{x \given M x \in Q}.
\]
\end{itemize}
Note that the closure in the definition of the conic hull can be omitted if $P$ is a polytope. Additionally, the polar cone is defined not only for cones but for arbitrary convex polyhedra $P$.

The goal is to express the epigraphs of $\eta$ and $\xi$ in terms of the game data $(A, B, S, T)$ for $\GG$, using a finite sequence of the operations listed above.

As already mentioned, a point $(y,r)$ belongs to $\epi \eta$ if and only
$$ y \in T \quad\text{and}\quad \forall x \in S:\; r \geq y^{\top} A^{\top} x\text{.}$$
This can be expressed equivalently as
$$ y \in T \quad\text{and}\quad \forall z \in S \times \Set{1}:\; \begin{pmatrix}
    Ay\\
    -r
\end{pmatrix}^{\top} z \leq 0 \text{.}$$
Using the concept of polar cone this can be written as
$$ y \in T \quad\text{and}\quad \begin{pmatrix}
    A & \phantom{-}0\\
    0 & -1
\end{pmatrix}  
\begin{pmatrix}
    y\\
    r
\end{pmatrix} =  \begin{pmatrix}
    Ay\\
    -r
\end{pmatrix} \in (S \times \Set{1})^*\text{.}$$
Using the concept of the inverse image under a linear transformation, the epigraph of $\eta$ can be expressed as
\begin{equation}
    \epi \eta = (T \times \R) \cap \begin{pmatrix}
        A & \phantom{-}0\\
        0 & -1
    \end{pmatrix}^{-1} [(S \times \Set{1})^*].
    \label{eq:prep1}
\end{equation}
Similarly, we obtain
\begin{equation}
    \epi \xi = (S \times \R) \cap \begin{pmatrix}
        B^{\top} & \phantom{-}0\\
        0\phantom{^{\top}} & -1
    \end{pmatrix}^{-1}[(T \times \Set{1})^*].
    \label{eq:prep2}
\end{equation}
P-representations of the epigraphs can now be easily obtained as compositions the operations involved in \eqref{eq:prep1} and \eqref{eq:prep2}. This can be achieved, for instance, using \textit{bensolve tools}. Once a P-representation of the epigraphs is found, the vertices can be computed by solving a polyhedral projection problem, which can also be accomplished with \textit{bensolve tools}.

\bigskip
\noindent
\textbf{Via vector linear programming:} 
We now want to show how the vertices of the epigraphs can be obtained directly using a vector linear programming solver. Let us first consider the particular case of unconstrained games. 
Consider the multiobjective linear programs
\leqnomode
	\begin{gather}\label{molp1}\tag{MOLP$_1$}
	  \max A^{\top} x 
         \text{ s.t. } x \in S\text{,}
	\end{gather}
\reqnomode
and 
\leqnomode
	\begin{gather}\label{molp2}\tag{MOLP$_2$}
	  \max B y 
         \text{ s.t. } y \in T\text{.}
	\end{gather}
\reqnomode
The linear programs~\eqref{p1} and~\eqref{p2} are simply weighted-sum scalarizations. A point \( x \in S \) is a \emph{weakly efficient point} of~\eqref{molp1} if and only if there exists \( y \in T \) such that \( x \) is an optimal solution of~\eqref{p1}. It is essential that \( T \) be the probability simplex, since the ordering cone of~\eqref{molp1} is \( \mathbb{R}^m_+ \). In this case, \( T \) serves as a base of the negative polar cone of the ordering cone—which is again \( \mathbb{R}^m_+ \).

A slight extension of this idea can also be applied to constrained games. Moreover, we observe that multi-objective linear programming duality yields the epigraphs of the optimal value functions \( \xi \) and \( \eta \).

We observe that $\epi \eta$ is the extended image of the geometric dual of the vector linear program
\leqnomode
	\begin{gather}\label{vlp1a}\tag{VLP$_1$}
	  \textstyle\max_D \begin{pmatrix}
	  	A^{\top} x \\
		0
	  \end{pmatrix} \text{ s.t. } x \in S\text{,}
	\end{gather}
\reqnomode
where
$$ D \coloneqq (-T \times \Set{-1})^* \subseteq \R^n \times \R $$
is the order cone, and the geometric duality parameter vector $d \in \Int D$ is the $(n+1)$-th unit vector.
$$ $$
Likewise, $\epi \xi$ is the extended image of the geometric dual of the vector linear program
\leqnomode
\begin{gather}\label{vlp2a}\tag{VLP$_2$}
	  \textstyle\min_C \begin{pmatrix}
	  	B y \\
		0
	  \end{pmatrix} \text{ s.t. } y \in T\text{.}
\end{gather}
\reqnomode
$$ C \coloneqq (-S \times \Set{-1})^* \subseteq \R^m \times \R\text{,} $$
and the geometric duality parameter vector $c \in \Int C$ is the $(m+1)$-th unit vector.

This connection can be seen as follows: It is known from geometric duality \cite{HeyLoe08,LoeWei17} that the extended image of the geometric dual of a vector linear program 
\leqnomode
	\begin{gather}\label{vlp}\tag{VLP}
	  \textstyle\max_D P x \text{ s.t. } x \in S\text{,}
	\end{gather}
\reqnomode
with geometric duality parameter $d \in \Int D$, $d_q = 1$ and objective matrix $P \in \R^{q \times m}$ is the epigraph of the following optimal value function
$\gamma : \R^{q-1} \to \R \cup \Set{\infty}$. For $y \in \R^{q-1}$, the value $\gamma(y)$ is the optimal value of the weighted sum scalarization of $\eqref{vlp}$ with weight vector 
$$ w = \left(y_1,\dots,y_{q-1},1-\sum_{i=1}^{q-1} d_i y_i\right) \in -D^\circ.$$
If $w \not\in -D^\circ$, we set $\gamma(y) = \infty$.
Since $q=n+1$ and $d$ is the $(n+1)$-th unit vector, the weight vector is of the form 
$w = (y,1)^{\top}$. Moreover, we have
$$-D^\circ = \cone(T \times \Set{1}).$$
Therefore, the conditions on $w$ translate to $w = (y,1)^{\top}$, $y \in T$ and the weighted sum scalarization of \eqref{vlp} is just the linear program \eqref{p1} introduced in the previous section.

\section{Numerical experiments}\label{sec:numeric}

In this work, whenever we solve constrained or unconstrained bi-matrix games, we employ the method based on polyhedral calculus~\cite{CirLoeWei19} and polyhedral projection introduced in the previous section, implemented using the software \textit{bensolve tools}~\cite{bt, CirLoeWei19}. This approach serves as a proof of concept, and we note that more efficient methods could be developed—for example, along the lines of the techniques for unconstrained games surveyed in~\cite{vStengel21}.

We begin with an example that illustrates the drawbacks of reformulating a constrained game as an unconstrained one, as proposed, for instance, in~\cite[Section 4]{MenZha14}, and compare it with directly solving the constrained game.

\begin{ex} Consider a constrained bi-matrix game given by the matrices  
\[
A = \begin{pmatrix}
-2 & \m 1 \\
-4 & -2
\end{pmatrix}, \quad
B = \begin{pmatrix}
-3 & \m 2 \\
\m 4 & -2
\end{pmatrix}
\]
and the feasible sets $S$ and $T$ both are the convex hulls of the columns of the matrix
\[
U \coloneqq \begin{pmatrix}
-2 & \m 2 & -1 & \m 1 & -1 & -2 & \m 1 & \m 2 \\
-1 & -1 & -2 & -2 & \m 2 & \m 1 & \m 2 & \m 1
\end{pmatrix}.
\]
We computed $5$ extremal Nash equilibria for the constrained game in approximately $0.1$ seconds. The constrained game can be reformulated as an unconstrained game by setting
\[
A' = U^{\top} A U, \qquad B' = U^{\top} B U.
\]
While the original constrained game has size \(2 \times 2\), the reformulated game is of size \(8 \times 8\). In this reformulation, we computed $148$ extremal Nash equilibria in about $20$ seconds.    
\end{ex} 

In all the following examples, the matrices \( A \in \mathbb{Z}^{m \times n} \) and \( B \in \mathbb{Z}^{m \times n} \) are generated as products of low-rank integer matrices: \( A = M_A N_A \) and \( B = M_B N_B \), where \( M_A \in \mathbb{Z}^{m \times r_A} \), \( N_A \in \mathbb{Z}^{r_A \times n} \), \( M_B \in \mathbb{Z}^{m \times r_B} \), and \( N_B \in \mathbb{Z}^{r_B \times n} \). The entries of each matrix are independently sampled from their own integer range: \( M_A \) from \([a(M_A), b(M_A)]\), \( N_A \) from \([a(N_A), b(N_A)]\), \( M_B \) from \([a(M_B), b(M_B)]\), and \( N_B \) from \([a(N_B), b(N_B)]\). This construction ensures that \( A \) and \( B \) have ranks at most \( r_A \) and \( r_B \), respectively. 

After generating a low-rank game as described, we check whether condition~\eqref{ass:A} is satisfied. If this is the case—which held in most of our experiments—we proceed to generate the corresponding low-rank game and solve it. 

If condition~\eqref{ass:A} is not satisfied, one can adjust the parameter \( t \) in the singular value decomposition and try again. However, the function
\[
t \mapsto \rank(A + tB)
\]
generically attains its maximum, as the following argument shows.\footnote{The authors thank Professor Raman Sanyal (Goethe University Frankfurt) for pointing this out.}

The determinants of the \( k \times k \) submatrices of \( A + tB \) are univariate polynomials in \( t \). The matrix \( A + tB \) has rank at most \( k \) if all such polynomials vanish. Assuming that at least one of these polynomials is not identically zero, they can vanish only for finitely many values of \( t \). Thus, for generic \( t \), the matrix \( A + tB \) has \( k \) linearly independent columns.

We observed in Proposition~\ref{prop:52a} that the optimal value functions of the original and reduced games coincide. These values represent the payoffs for both players. A \emph{payoff plot} illustrates the possible payoffs corresponding to maximal pairs of Nash faces, using distinct colors to distinguish each pair. More precisely, for every maximal pair of Nash faces \((F, G)\), we display the set
\[
\bigl\{\,(\eta(y), \xi(x))^{\top} \mid (x, \xi(x))^{\top} \in F, \; (y, \eta(y))^{\top} \in G \,\bigr\}.
\]
By Proposition~\ref{prop:52a}, together with our main result, Theorem~\ref{thm:main}, we see that the original and reduced games have exactly the same payoff plots. However, since the one-to-one correspondence in the main result is restricted to maximal pairs of Nash equilibria, the number of extreme Nash equilibria may differ.

\begin{ex}\label{ex:ex03-paper}
%
As a concrete example, we set \( m = n = 100 \) and \( r_A = r_B = 2 \). The matrices \( A = M_A N_A \) and \( B = M_B N_B \) are constructed with
\[
M_A, M_B \in \mathbb{Z}^{100 \times 2} \text{ sampled from } [-2, 2], \quad
N_A, N_B \in \mathbb{Z}^{2 \times 100} \text{ sampled from } [1, 4].
\]
This ensures that both \( A \) and \( B \) have rank at most 2. If condition~\eqref{ass:A} is satisfied, we obtain a reduced game of size \( 4 \times 4 \), which we solve be the described methods. Figure~\ref{fig:ex03-paper} shows the payoff plot for a concrete sample.
\end{ex}

\begin{figure}
    \centering
    \includegraphics[width=.6\textwidth]{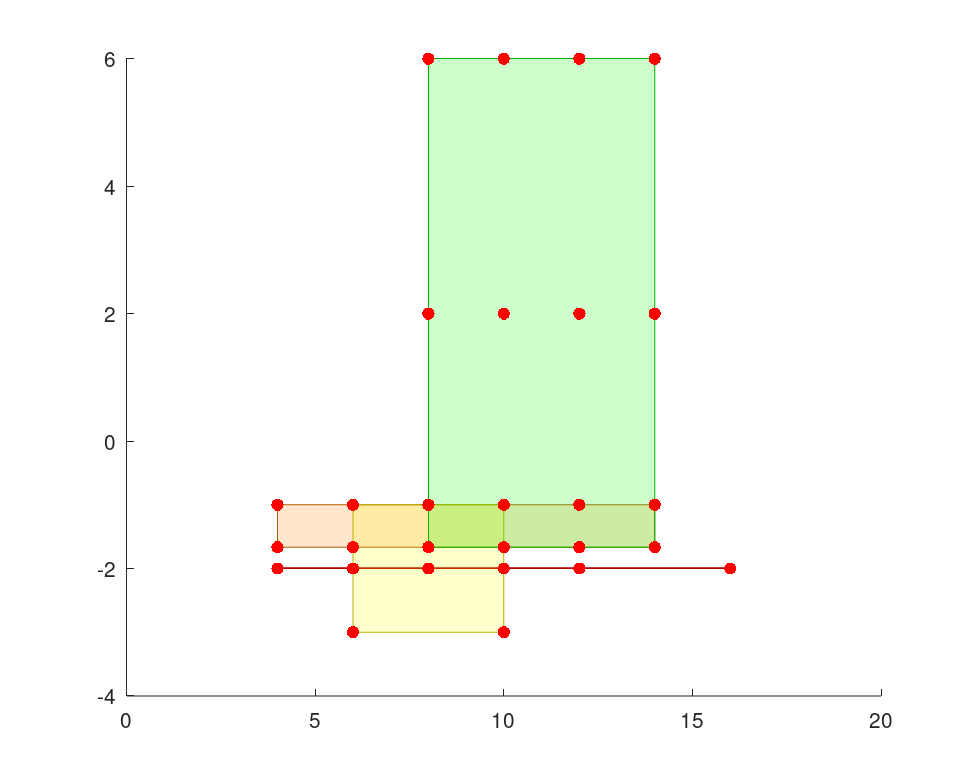}
    \caption{The figure shows the payoff plot for Example \ref{ex:ex03-paper}. There are three rectangles and one line segment. Each of these four objects corresponds to a maximal pair of Nash faces and shows the possible payoffs for both players. The red points indicate the payoffs at extremal Nash equilibria of the reduced game. The original $100 \times 100$ game with $\rank A = \rank B = 2$, which was not computed here, has the same payoff plot, i.e., the same four rectangular regions, but it is allowed to have a different number of extremal equilibria.
}
    \label{fig:ex03-paper}
\end{figure}

\begin{ex}\label{ex:ex04-paper}
Consider an unconstrained bimatrix game defined by the following $8 \times 9$ matrices of rank~$2$:




\[
A =
\begin{pmatrix}
  -3 &  -5 &  -2 &  -2 &  -2 &  -3 &  -4 &  -2 &  -2 \\
  -6 & -10 &  -4 &  -6 &  -8 &  -6 &  -4 &  -6 & -10 \\
  -3 &  -5 &  -2 &  -2 &  -2 &  -3 &  -4 &  -2 &  -2 \\
 \m 0 & \m 0 & \m 0 & \m 1 & \m 2 & \m 0 &  -2 & \m 1 & \m 3 \\
 \m 0 & \m 0 & \m 0 &  -1 &  -2 & \m 0 & \m 2 &  -1 &  -3 \\
  -3 &  -5 &  -2 &  -4 &  -6 &  -3 & \m 0 &  -4 &  -8 \\
 -12 & -20 &  -8 & -10 & -12 & -12 & -12 & -10 & -14 \\
 \m 3 & \m 5 & \m 2 & \m 2 & \m 2 & \m 3 & \m 4 & \m 2 & \m 2
\end{pmatrix}
\]

\[
B =
\begin{pmatrix}
 \m 0 & \m 0 &  -4 & \m 0 & \m 4 & \m 2 & \m 0 &  -4 & \m 4 \\
  -3 &  -3 &  -3 &  -3 & \m 3 & \m 3 & \m 3 &  -3 & \m 3 \\
  -1 &  -1 & \m 7 &  -1 &  -7 &  -3 & \m 1 & \m 7 &  -7 \\
 \m 1 & \m 1 &  -1 & \m 1 & \m 1 & \m 0 &  -1 &  -1 & \m 1 \\
 \m 0 & \m 0 & \m 4 & \m 0 &  -4 &  -2 & \m 0 & \m 4 &  -4 \\
 \m 1 & \m 1 & \m 3 & \m 1 &  -3 &  -2 &  -1 & \m 3 &  -3 \\
  -3 &  -3 &  -1 &  -3 & \m 1 & \m 2 & \m 3 &  -1 & \m 1 \\
  -1 &  -1 & \m 1 &  -1 &  -1 & \m 0 & \m 1 & \m 1 &  -1
\end{pmatrix}.
\]
We solved the original game and the reduced game for $t=1$, obtaining $16$ and $12$ extremal Nash equilibria, respectively. This demonstrates that the number of extremal Nash equilibria can differ. In both cases we computed four maximal Nash faces.
\end{ex}


\begin{ex}
We set \( m = n = 500 \) and \( r_A = r_B = 2 \). The matrices \( A = M_A N_A \) and \( B = M_B N_B \) are constructed with
\[
M_A, M_B \in \mathbb{Z}^{500 \times 2}, N_A, N_B \in \mathbb{Z}^{2 \times 500} \text{ sampled uniformly from } [-3, 3],
\]
We solved 10 random instances on a desktop computer with a 3.6 GHz clock speed. The average running time to solve one reduced game was approximately 10 seconds.
\end{ex}

\bibliography{ref}

\end{document}